\documentclass[preprint,12pt]{elsarticle}

\usepackage{amssymb}
\usepackage{mathrsfs}
\usepackage{amsmath}
\usepackage{amsthm}

\journal{Nuclear Physics B}

\begin{document}

\begin{frontmatter}



\title{Reinforcement Learning Framework For Stochastic Optimal Control Problem Under Model Uncertainty}


\author[author1]{JiaXuan Hou}
\address[author1]{School of Mathematical Science, Ocean University of China, China}

\author[author3]{LiFeng Wei \corref{cor1}}
\address[author3]{School of Mathematical Science, Ocean University of China, China}
\cortext[cor1]{Corresponding author}

\author[author2]{GuangChen Wang}
\address[author2]{School of Control Science  and Engineering,ShanDong University, China}

\begin{abstract}
We develop a continuous-time entropy-regularized reinforcement learning framework under model uncertainty. By applying Sion’s minimax theorem, we transform the intractable robust control problem into an equivalent standard entropy-regularised stochastic control problem, facilitating reinforcement learning algorithms. 
We establish sufficient conditions for the theorem’s validity and demonstrate our approach on linear-quadratic problems with uncertain model parameters following Bernoulli and uniform distributions.
\end{abstract}


\end{frontmatter}

\section{Introduction}
We consider a infinite horizon stochastic optimal control problem under model uncertainty. Within this framework the cost function has uncertain parameters $\theta$, which is used to represent different market conditions. The work of Hu and Wang \cite{hu2020maximum} addresses a stochastic recursive optimal control problem under model uncertainty. Specifically, \cite{hu2020maximum} derives the corresponding stochastic maximum principle and additionally investigates a linear-quadratic robust control problem. 

Since numerous applications of stochastic optimal control problems have been found in many areas such as economy, engineering, and biology, it has become important to know how to solve stochastic optimal control problems. Early foundational work, exemplified by Peng et al.\cite{peng1990general}, \cite{peng1992generalized}, \cite{peng1992stochastic} established theoretical results for stochastic optimal control, including maximum principles and dynamic programming principles. Lately, there has been a growing interest in using reinforcement learning(RL) to solve stochastic optimal control problems (see \cite{jia2023q} for more details).

Reinforcement learning and adaptive dynamic programming(ADP) provide a framework for data-driven, learning-based approaches to problems of stochastic optimal control problems (see articles \cite{sutton1998reinforcement} \cite{wang2020reinforcement} for more details)
. In 2020, Wang and Zhou \cite{wang2020reinforcement} proposed a continuous-time relaxed stochastic control formulation, which involves the differential entropy of distributions of actions to trade off the exploration and exploitation in reinforcement learning. Moreover, Wang et al. \cite{wang2020continuous} developed an RL method for an mean-variance portfolio selection. And Sun et al. \cite{jia2022policy} solves an two-person zero-sum differential game problem.\par
In this paper we establish a RL framework to solve stochastic optimal control problems under model uncertainty. Model uncertainty refers to the agent’s incomplete knowledge about the real environment dynamics. For example, in portfolio selections, the share market is characterized by different coefficients in the bull market and bear market. However, it is hard to predict the actual coefficients before investing. A natural idea is to measure the cost in a robust way. We define the cost function based on the most unfavorable condition for the investor, which ensures a minimal loss under the worst-case scenario. But currently, there remain challenges for classical entropy-regularized RL algorithms to solve stochastic optimal control problems under model uncertainty. Parametric uncertainty introduces uncertainty in the policy distribution, thus posing difficulties in solving for optimal policies.\par
To address this problem, we propose two cases of continuous-time entropy-regularized reinforcement learning (RL) framework under model uncertainty. Specifically, we consider an agent interacting with a family of potential models parameterized by an unknown variable $\theta \in \Theta$. The true model is unknown to the agent, who instead aims to achieve robust performance across all plausible scenarios. Notably, the robust cost is formulated as the supremum over a family of probability measures, making classical approaches for deriving the Hamilton–Jacobi–Bellman (HJB) equation difficult to apply. To overcome this challenge, we handle the derivative of the value function using Sion’s minimax theorem. Building on these results, we apply the entropy-regularized RL framework to a linear-quadratic (LQ) robust control problem. We show that the entropy-regularized RL framework retains solvability under model uncertainty and converges to the classical LQ formulation as the exploration weight tends to zero.\par
This paper is organized as follows. In section 2 we formally introduces the robust entropy-regularized RL problem and its formulation via controlled SDEs.In section 3 we present the associated Hamilton-Jacobi-Bellman(HJB) equation and the optimal control distribution under model uncertainty. Section 4 focuses on the linear-quadratic case, characterizing optimal Gaussian policies when the model coefficient follows a two-point and a uniform distribution respectively.Then we verifies their equivalence to classical solvable cases. Section 5 discusses the convergence to classical control problems in the zero-entropy limit and examines the cost of exploration. Section 6 concludes with a summary and future directions.

\section{Formulation Of Problem}
We now introduce the entropy-regularized relaxed stochastic control problem under model uncertainty and provide its motivation in the context of RL.\par
\textbf{Notation.}Throughout this paper, let $(\mathscr{F}_t)_{0 \leq t\leq T}$be the natural filtration generated by W augmented by the $\mathbb{P}$-null sets of$\mathscr{F}$.Denote by $\mathbb{R}^n$ the n-dimensional real Euclidean space,$\mathbb{R}^{n\times d}$ the set of $n\times d$ real matrices, and $\mathbb{S}_n$the set of symmetric $n \times n$ real matrices.
\par  
$\mathcal{M}^p(0,\infty;\mathbb{R})$ is the space of $\mathbb{R}$-valued $\mathscr{F}$-progressively measurable processes $u(t)$ on $[0, \infty]$ satisfying $$\mathbb{E}\left[\int_0^\infty |u(t)|^p dt \right] <\infty;$$
$\mathcal{H}^p(0,\infty;\mathbb{R})$ is the space of $\mathbb{R}$-valued $\mathscr{F}$-progressively measurable processes $z(t)$ on $[0,\infty)$ satisfying 
$$\mathbb{E}\left[\left(\int_0^\infty |z(t)|^2 dt \right)^{\frac p2}\right] <\infty;$$
$\mathcal{H}^{1,p}(0,\infty;\mathbb{R})$ is the space of $\mathbb{R}$-valued $\mathscr{F}$-progressively measurable processes $z(t)$ on $[0,\infty)$ satisfying 
$$\mathbb{E}\left[\left(\int_0^\infty |z(t)| dt \right)^p\right] <\infty.$$\par
\subsection{Classical stochastic optimal control problem under model uncertainty}
Firstly we introduce the definition of admissible control in classical problem. Assume $U$ is a given nonempty convex subset of $\mathbb{R}$. And $u:[0,\infty)\times \Omega \rightarrow U$ is said to be an admissible control if $u \in \mathcal{M}^p(0,\infty;\mathbb{R}^k)$. The set of admissible controls is indicated by $\mathcal{U}[0,\infty).$\par
Assuming an agent in the market can choose a control $u(t)$ from $\mathcal{U}[0,\infty)$ to obtain some SDE on $[0,\infty)$. However, sometime actual drift and diffusion coefficients is unknown due to the model uncertainty. Instead, a set of coefficients that could potentially occur in market is explicit. In this case, the classical stochastic control problem is described by
\begin{equation}
    dx^{cl}_{\theta}(t) = b_{\theta}(t,x^{cl}_{\theta}(t),u_t)dt+\sigma_{\theta}(t,x^{cl}_{\theta}(t),u_t)dW(t), \ t>0 \label{clasde}
\end{equation}
with $x_\theta^{cl}(0)=x, \ x<|M|$.
We have $\theta \in \Theta$ and $\Theta$ is a locally compact complete separable space with distance $\mu$. The corresponding cost is given by the $y_\theta(0)$ term of the following SDE on $[0,\infty):$
\begin{equation}
    \begin{aligned}
        y_\theta^{cl}(0) = \mathbb{E}\left[\int_0^\infty e^{-\rho t}f_\theta(t,x_\theta(t),u)\  dt|x_0 = x\right].
    \end{aligned}
\end{equation}
where $\rho>0$ is the discount rate which is subject to be determined, and $b_\theta:[0,\infty)\times\mathbb{R}\times U\rightarrow\mathbb{R}$, $\sigma_\theta :[0,\infty)\times\mathbb{R}\times U\rightarrow\mathbb{R}$, $f_\theta:[0,\infty)\times\mathbb{R}\times U\rightarrow\mathbb{R}$ are Borel measurable functions. Due to the model uncertainty, the cost function is defined by
$$J^{cl}(u)= \sup_{Q\in \mathcal{Q}}\int_{\Theta} y_\theta^{cl}(0)Q(d\theta),$$
where $Q$ is a set of probability measure on $(\Theta, \mathcal{B}(\Theta))$.  \par
In this paper, we make the following assumptions.\\
(H2.1)$|b_{\theta}(t,x,u)-b_{\theta}(t,x',u')|+|\sigma_{\theta}(t,x,u)-\sigma_{\theta}(t,x',u')| \leq L(|x-x'|+|u-u'|), \forall x,x'\in \mathbb{R}, u,u'\in U,\theta\in\Theta,t\in[0,+\infty],  \forall x,x'\in \mathbb{R}, u,u'\in U,\theta\in\Theta,t\in[0,+\infty]$,\\ $|b_{\theta}(t,0,0)|+|\sigma_{\theta}(t,0,0)|+|f_{\theta}(t,0,0)|\leq L$, for any $t\geq0$ and some positive constant $L$.\\
(H2.2)For each $u\in \mathcal{U},\ x\in \mathbb{R} $, $b_\theta(\cdot,x,u) \in \mathcal{H}^{1,p}(0,\infty; \mathbb{R})$ and $ \sigma_\theta(\cdot,x,u) \in \mathcal{H}^p(0,\infty;\mathbb{R})$ for some $p\geq2$.\\
(H2.3)For each $N>0$, there exists a modulus of continuity $\bar{\omega_N}:[0,\infty)\rightarrow[0,\infty) $ such that $|l_\theta(t,x,u)-l_{\theta'}(t,x,u)|\leq \bar{\omega}_N(\mu(\theta,{\theta'}))$ for any $t\in [0,T], |x|,|u|\leq N, \theta, \theta'\in \Theta$, where $l_\theta$ is $b_\theta, \sigma_\theta,f_\theta$ and their derivatives in $(x,u)$.\\
(H2.4)$\mathcal{Q}$ is a weakly compact and convex set of probability measure on $(\Theta, \mathcal{B}(\Theta)).$

The aim of the control is to achieve the minimum robust cost, thus we define the following value function: 
\begin{equation}
\begin{aligned}
    V^{cl}(x) = \inf_{u\in \mathcal{U}} \sup_{\mathbf{Q}\in\mathcal{Q}}\int_{\Theta}y^{cl}_\theta(0)\ \mathbf{Q}(d\theta)
\end{aligned}    
\end{equation}
In the classical setting the model is fully explicit. Applying the dynamic programming principle, the optimal control can be derived as a deterministic mapping from the current state to the action space $U$.

\subsection{Exploratory stochastic optimal control problem under model uncertainty}
In the context of reinforcement learning, the absence of explicit models necessitates the adoption of dynamic learning approaches. Based on the method in \cite{wang2020reinforcement},for each $\theta \in\Theta$ we introduce the exploratory version of the state dynamics. Define
$$ \widetilde{b}_{\theta}(t,X_{\theta}^{\pi}(t),{\pi}_t) := \int_{U}b_{\theta}(t,X_{\theta}^{\pi}(t),u)\pi_{t}(u) du\quad,$$
$$\widetilde{\sigma}_{\theta}(t,X_{\theta}^{\pi}(t),{\pi}_t) := \sqrt{\int_{U}\sigma_{\theta}^2(t,X_{\theta}^{\pi}(t),u)\pi_{t}(u) du}\quad. 
$$
Here drift $\widetilde{b}_{\theta}(t,X_{\theta}^{\pi}(t),{\pi}_t)$ and volatility $\widetilde{\sigma}_{\theta}(t,X_{\theta}^{\pi}(t),{\pi}_t)$ are incorporated into the following exploratory formulation of the state equation:
\begin{equation}
\begin{cases}
    dX^{\pi}_{\theta}(t) = \widetilde{b}_{\theta}(t,X^{\pi}_{\theta}(t),\pi_t)dt+\widetilde{\sigma}_{\theta}(t,X^{\pi}_{\theta}(t),\pi_t)dW(t), \quad t\in[0,\infty) \\
    X^{\pi}_{\theta}(0) = x,\quad x\in [-M,M]\ (M>0).\label{exploratory sde}
\end{cases}
\end{equation}
In this formulation, the control process $u_t$ is randomized, leading to a distributional control process over the control space $U$ whose density function is denoted by $\pi=\{\pi_t, t>0\}.$ Note that the control $u$ is no longer a mapping from $[0,\infty)\times \Omega$ to $U$, but a point on $U$ sampled from the distribution $\pi$. \par
Thus we can define the following space:\par
For any distribution $\pi$, $\mathcal{G}^p(U;\mathbb{R})$ is the space of $\mathbb{R}$-valued $\mathcal{B}(U)$-measurable functions $f(u)$ satisfying $$ \int_U|f(u)|^p \pi(u) du<\infty;$$ \par
And we add the followig assumption:\\
(H2.5)For each $t\in \mathbb{R},\ x\in \mathbb{R} $, $b_\theta(t,x,\cdot), \sigma_\theta(t,x,\cdot), $  $f_\theta(t,x,\cdot)  \in \mathcal{G}^p(U; \mathbb{R}).$\par

To capture the agent's exploratory behavior, the control is modeled as a randomized distribution $\pi_t(u)$. Due to the absence of complete model knowledge, a regularization term is introduced to promote sufficient exploration. Thus, similar to \cite{wang2020reinforcement}, we add Shannon's differential entropy to measure the level of exploration:
$$\mathcal{H}(\pi):= -\int_U \pi(u)\ln\pi(u)du,\quad \pi \in \mathcal{P}(U).$$
For convenience, we define
$$\widetilde{f}_\theta(t,X^\pi_\theta(t),{\pi}_t(u)) =\int_Uf_\theta(t,X^\pi_\theta(t),u){\pi}_t(u) du,$$
Then with \eqref{exploratory sde}, for each $\theta \in \Theta$ we can introduce the exploratory form of corresponding cost :
\begin{equation}
    \begin{aligned}
        y^\pi_\theta(0) = \mathbb{E}\left[\int_0^\infty e^{-\rho t}\left( \widetilde{f}_\theta(t,X^\pi_\theta(t),\pi_t(u))- \alpha\mathcal{H}(\pi_t) \right)dt\right],\label{ypi0}
    \end{aligned}
\end{equation}    
where exogenous parameter $\alpha>0$ denote the trade-off between exploitation and exploratory. Considering the model uncertainty, the exploratory formulation of cost function is defined by:
\begin{equation}
    J(\pi)= \sup_{Q\in \mathcal{Q}}\int_{\Theta} y^\pi_\theta(0)Q(d\theta).\label{Jpi}
\end{equation}
We will minimize \eqref{Jpi} by dynamic programming. For that we introduce the value function of this entropy-regularized relaxed stochastic control problem under model uncertainty
\begin{equation}
\begin{aligned}
    &V(x) :=\inf_{\pi\in \mathcal{A}(x,\theta)}\sup_{\mathbf{Q}\in\mathcal{Q}} \int_{\Theta}y^\pi_\theta(0)\ \mathbf{Q}(d\theta) \label{V(x)1}.
\end{aligned}    
\end{equation}
The admissible control set $\mathcal{A}(x,\theta)$ may depends on $x$ and $\theta$ in general. We define it as follows. Let $\mathcal{B}(U)$ be the Borel algebra on $U$. A distributional control process ${{\pi}_t,t\geq0}$ belongs to $\mathcal{A}(x,\theta)$,if: \par
(i) for each $t\geq 0,{\pi}_t \in \mathcal{P}(U)$ a.s.;\par
(ii) for each $A\in \mathcal{B}(U)$,$\{\int_A {\pi}_t(u) du, t\geq 0\}$ is $\mathcal{F}_t$-progressively measurable;\par
(iii) for each $t\geq 0,\ \int_U|u|^k \pi_t(u)du < \infty,  k =1,2,...$;\par
(iv) the state dynamic\eqref{exploratory sde} has a unique strong solution $X^\pi = \{X^\pi_t, t\geq 0\}$ if ${\pi}_t$ is applied;\par
(v)The right hand side of \eqref{ypi0} is bounded.\\
And we add this assumption:\par
(H2.6) With $\{X_\theta(t), t\geq 0\}$ solving \eqref{exploratory sde}, $X_t\ \widetilde{b} _\theta(t,X_\theta(t), \pi)+\frac12 \widetilde{\sigma}_\theta^2(t, X_\theta(t),\pi)\\ \leq \frac{-\beta}{2}X_\theta^2(t)$ for $\forall\pi \in \mathcal{A}(x,\theta)$. 
\newtheorem{lemma}{Lemma}
\begin{lemma}
    Asuume that (H2.1)-(H2.6) hold. Then, $\theta\rightarrow X_\theta^\pi(t)$ is continuous with respect to $\theta$.\label{lemmaxtheta}
\end{lemma} 
\begin{proof}
    For any given $\varepsilon > 0$, we define these stopping time by
$$\tau_{1\varepsilon}:= \inf\{t\geq 0: \mathbb{E}|X_{\theta}(t)|^2<\frac{\varepsilon}{2}\},$$
$$\tau_{2\varepsilon}:= \inf\{t\geq 0: \mathbb{E}|X_{\theta'}(t)|^2<\frac{\varepsilon}{2}\},$$
and $\tau_{\varepsilon} = \tau_{1\varepsilon} \lor \tau_{2\varepsilon} $.\par Denote
$A(t)= \int_0^1\partial_xb_{\theta'}(t,X_{\theta'}^\pi(t)+\lambda(X_{\theta}^\pi(t)-X_{\theta'}^\pi(t)))d\lambda;$\\ 
$C(t)= b_\theta(t,X_{\theta'}^\pi(t)))-b_{\theta'}(t,X_{\theta'}^\pi(t)));$ 
$D(t)= \sigma_\theta(t,X_{\theta'}^\pi(t)))-\sigma_{\theta'}(t,X_{\theta'}^\pi(t)));$ 
$E(t)= \sigma_\theta(t,X_{\theta}^\pi(t)))+\sigma_{\theta'}(t,X_{\theta'}^\pi(t)));$ $B(t)= \int_0^1\partial_x\sigma_{\theta'}(t,X_{\theta'}^\pi(t)+\lambda(X_{\theta}^\pi(t)-X_{\theta'}^\pi(t)))d\lambda;$ \\
$$F(t)=\sqrt{\int_U\sigma^2_\theta(s,X^\pi_\theta(t))\pi(u)du}+\sqrt{\int_U\sigma^2_{\theta'}(s,X^\pi_{\theta'}(t))\pi(u)du}.$$
Thus, the process $\alpha$ satisfies the following SDE:
\begin{equation}
\begin{aligned}
\alpha(t)&= \int_0^{\tau_{\varepsilon}} \int_U (A(t)\alpha(t)+C(t) )\pi(u)du\ dt \\ &+
\int_0^{\tau_{\varepsilon}}\frac{\int_U \left(B(t)\alpha(t)+D(t)\right)E(t)\pi(u)du}{F(t)}d W_t.
\end{aligned}
\end{equation}
For the following analysis, we recall the following result from \textbf{Lemma A.1} in  \cite{hu2020maximum}.
It states that if (H2.1) and (H2.2) holds, then the SDE \eqref{clasde} satisfies that
$$\mathbb{E}\left[\sup_{0\leq t\leq T}|x_{\theta}(t)|^p\right] \leq C(L,T,p)\mathbb{E}\left[|x_0|^p + \left(\int_0^T|b_{\theta}(t,0)|\right)^p + \left(\int_0^T|\sigma_{\theta}(t,0)|^2\right)^{\frac p2}\right].$$

On $[0,\tau_{\varepsilon}]$ applying the above inequality yields that 
\begin{equation}
\begin{aligned}
&\mathbb{E}[\sup_{0\leq t\leq \tau_{\varepsilon}}|X_\theta^\pi(t)-X_{\theta'}^\pi(t)|^2]\\ & \leq C(L,\tau_{\varepsilon})\mathbb{E}\left[ \int_0^{\tau_{\varepsilon}}\left(\int_UC(t)\pi(u)du\right)^2+ \left(\frac{\int_U E(t)D(t)\pi(u)du}{F(t)}\right)^2 dt\right]\\ &
\leq C(L,\tau_{\varepsilon})\mathbb{E}\left[ \int_0^{\tau_{\varepsilon}}\int_U C(t)^2 \pi(u)du+ \left(\int_U\frac{ E(t)^2\pi(u)}{F(t)^2} du \right)^{\frac12} \left(\int_U D(t)^2 \pi(u)du \right)^{\frac12}  dt\right]\\ & \leq C(L,\tau_{\varepsilon})\mathbb{E}\left[ \int_0^{\tau_{\varepsilon}}\int_U C(t)^2 \pi(u)du+ \sqrt{2} \left(\int_U D(t)^2 \pi(u)du \right)^{\frac12}  dt\right]\\ &\leq  C(L,\tau_{\varepsilon},x_0,\pi)(\bar{\omega}_N^2(\mu(\theta,\theta'))+N^{-1}+N^{-\frac12})
\end{aligned}  
\end{equation}
for any $N>0$. For the detailed derivation of the final inequality, we refer the reader to \cite{hu2020maximum}. Then we could get that
$$\lim_{\epsilon\rightarrow0}\sup_{\mu(\theta,\theta')\leq\epsilon}\mathbb{E}\left[\sup_{0\leq t\leq \tau_{\varepsilon}}|X_\theta(t)-X_{\theta'}(t)|^2\right]=0.$$
Now we set
$$V(t, X_\theta(t))= \frac12 X_\theta(t)^2.$$
Given the assumption (H2.6), we can directly derive that $V(t, X_\theta(t))$ is a Lyapunov function for \eqref{exploratory sde}.
Then for $t>\tau_{\varepsilon}$, from \textbf{Lemma \ref{appendixB}} in Appendix we immediate have
$$\mathbb{E}\left[|X_{\theta}(t)-X_{\theta'}(t)|^2\right]\leq\mathbb{E}\left[|X_{\theta}(t)|^2+|X_{\theta'}(t)|^2\right] \leq \varepsilon. $$
This proves that $X^{\pi}_{\theta}$ is continuous with respect to $\theta$.
\end{proof}
\begin{lemma}
Asuume that (H2.1)-(H2.6) hold. Then, $\theta\rightarrow y_\theta^\pi(0)$ is continuous with respect to $\theta$.    
\end{lemma}
\begin{proof}
    \begin{equation}
    \begin{aligned}
        |y_\theta(0)-y_{\theta'}(0)|&\leq \mathbb{E}\left[\int_0^\infty  e^{-\rho t}    |f_\theta(t,X^\pi_\theta(t),u)-f_{\theta'}(t,X^\pi_{\theta'}(t),u)|\pi_t(u) du \ dt\right]
    \end{aligned}
\end{equation}
Since $f_\theta(t,X^\pi_\theta(t),u)-f_{\theta'}(t,X^\pi_{\theta'}(t),u)= f_\theta(t,X^\pi_\theta(t),u)- f_\theta(t,X^\pi_{\theta'}(t),u)+f_\theta(t,X^\pi_{\theta'}(t),u)-f_{\theta'}(t,X^\pi_{\theta'}(t),u)$, with assumption (H2.1), it holds that
\begin{equation}
    \begin{aligned}
        |y_\theta(0)-y_{\theta'}(0)|&\leq \mathbb{E}[\int_0^\infty \int_U e^{-\rho t}(L|X_\theta^\pi(t)-X_{\theta'}^\pi(t)|\\ &+|f_\theta(t,X^\pi_{\theta'}(t),u)-f_{\theta'}(t,X^\pi_{\theta'}(t),u)|)\pi_t(u) du\  dt].
    \end{aligned}
\end{equation}
For any $N>0$, we could get that
\begin{equation}
\begin{aligned}
    &\mathbb{E}[\int_0^\infty \int_U e^{-\rho t}|f_\theta(t,X^\pi_{\theta'}(t),u)-f_{\theta'}(t,X^\pi_{\theta'}(t),u)|\ \pi_t(u) du\  dt] \\ &\leq \mathbb{E}\left[ \int_0^\infty \int_U e^{-\rho t}(\bar{\omega}_N(\mu(\theta, \theta'))+2L(1+|X_\theta(t)|+u)(I_{\{|X_\theta(t)|\geq N\}}+I_{\{u\geq N\}}))\pi(u) du\ dt\right]\\ &\leq
    C(L,x_0, \pi, \rho)(\bar{\omega}_N(\mu(\theta, \theta'))+N^{-1})\ \ \forall N>0 .
\end{aligned}
\end{equation}
which implies that
$$\lim_{\epsilon\rightarrow0}\sup_{\mu(\theta,\theta')\leq\epsilon}\mathbb{E}[\int_0^\infty \int_U e^{-\rho t}|f_\theta(t,X^\pi_{\theta'}(t),u)-f_{\theta'}(t,X^\pi_{\theta'}(t),u)|\ \pi_t(u) du\  dt]=0.$$
The moment exponential stability of $X_\theta(t)$ on $[0,\infty)$ yields that $\int_0^\infty \int_U e^{-\rho t}(L|X_\theta^\pi(t)-X_{\theta'}^\pi(t)|\pi(u))du\ dt <\infty$. Thus appealing to \textbf{Lemma\ref{lemmaxtheta}} the proof is complete.
\end{proof}

\section{HJB Equation and Optimal Distributions}
In this section, we derived the HJB equation of the uncertain model stochastic optimal control problem. Then with the help of Sion's minimax theorem, we establish a general RL procedure for solving optimization problem under model uncertainty.\par
Now we are going to discuss the HJB equation. 
Firstly we introduce this subset of $\mathcal{Q}$: for each $\pi\in \mathcal{P}(U)[0,\infty)$,
$$\mathcal{Q}^\pi = \{Q\in \mathcal{Q}|J(\pi)= \sup_{Q\in \mathcal{Q}}\int_{\Theta}y_\theta(0)Q(d\theta)\}.$$
To ensure the applicability of the variational method, we need the following lemma.\par
\begin{lemma}
    Suppose that (H2.1)-(H2.6) hold. Then the set $ \mathcal{Q}^\pi$ is nonempty for each $\pi \in \mathcal{P}(U).$
\end{lemma} 
\begin{proof}
    By the definition $J(\pi)$, there exist a sequence $Q^N \in \mathcal{Q}$ so that 
$$J(\pi)-\frac1N \leq \int_\Theta y_\theta(0)Q^N(d\theta)\leq J(\pi).$$
Note that $\mathcal{Q}$ is weakly compact. Then choosing a subsequence if necessary, we could find a $Q^\pi \in \mathcal{Q}$ such that $Q^N$ converges weakly to $Q^\pi$. From assumptions we can deduce that the function $\theta \rightarrow y_\theta(0)$ is continuous and bounded. It follows that
$$J(\pi) \geq \int_\Theta y_\theta(0) Q^\pi(d\theta) = \lim_{N\rightarrow \infty}\int_\Theta y_\theta(0) Q^N(d\theta)\geq J(\pi),$$
which ends the proof.
\end{proof}
For the sake of notation, we denote
\begin{equation}
\widetilde{L}(t,X_{\theta}^\pi(t),{\pi}_t)= \widetilde{f}_{\theta}(t,X_{\theta}^{\pi}(t),\pi^{\theta}_t(u)) -\alpha\mathcal{H}({\pi}_t).\label{L}
\end{equation}
Applying the Bellman's principle of optimality to \eqref{V(x)1}, and invoking lemma 3.1, we conclude that for the optimal control $\pi^*$, there exists a probability measure $Q\in \mathcal{Q}^{\pi^*}$  such that
\begin{equation}
\begin{aligned}
    &V(x) =\inf_{\pi\in \mathcal{A}(x,\theta)}\int_{\Theta}\mathbb{E}[\int_0^{t} e^{-\rho s} \widetilde{L}_\theta(s,X_\theta^\pi(s),{\pi}_s) \ ds + e^{-\rho t}V(X_\theta^\pi (t))\ ]\ Q(d\theta) ,\  s>0.
\end{aligned}    
\end{equation}
Proceeding with standard arguments, we deduce that \eqref{V(x)1} satisfies the HJB equation
\begin{equation}
    \begin{aligned}
        &\rho V(x)=\inf_{\pi \in \mathcal{A}(x,\theta)} \sup_{Q\in \mathcal{Q}}\int_\Theta \widetilde{L}_\theta(0,x,{\pi}_0) + V'(x)\widetilde{b}_\theta(0,x,{\pi}_0)\\ &+\frac12 V''(x)\widetilde{\sigma}_\theta^2(0,x,{\pi}_0
        ) \ Q(d\theta).
    \end{aligned}
\end{equation}

\newtheorem{theorem}{Theorem}
\begin{theorem}
    Assume that the condition (H2.4) holds, and control distribution $\pi(u)$ satisfied all the requirements above. Then we can derive that 
\begin{equation}
    \begin{aligned}
        \rho V(x)&=\sup_{Q\in \mathcal{Q}}\inf_{\pi \in \mathcal{A}(x,\theta)} \int_\Theta \widetilde{L}_\theta(0,x,{\pi}_0) + V'(x)\widetilde{b}_\theta(0,x,{\pi}_0)\\ &+\frac12 V''(x)\widetilde{\sigma}_\theta^2(0,x,{\pi}_0
        )\  Q(d\theta).\label{hjbthe}
    \end{aligned}
\end{equation}
\end{theorem}

\begin{proof}
    Denote \\$g_\theta^\pi(0)= \int_{\Theta}\widetilde{f}_\theta(0,x,{\pi}_0) + V'(x)\widetilde{b}_\theta(0,x,{\pi}_0)+\frac12 V''(x)\widetilde{\sigma}_\theta^2(0,x,{\pi}_0) \ Q (d\theta)$. With the help of Holder inequality, we can deduce that
\begin{small}
\begin{equation}
\begin{aligned}
    &|g_\theta^\pi(0)-g_\theta^{\pi'}(0)| \leq \\ & \left(\int_\Theta\int_u\left( f_\theta(0,x,u)+ V'(x)b_\theta(0,x,u) du+\frac12 V''(x)\sigma_\theta^2(0,x,u)\right)^2|\pi(u)-\pi'(u)|duQ(d\theta)\right)^{\frac12}\\ & \times\left(\int_\Theta\int_u |\pi(u)-\pi(u)'|du Q(d\theta)\right)^{\frac12}.
\end{aligned}
\end{equation}
\end{small}
Using the assumption (H2.5), a directly result yields that $\pi \rightarrow  g_\theta^\pi(0) $ is continuous. 
As established by \cite{gora1990differential}, it follows that the function $-\mathcal{H}(\pi)$ is lower semicontinuous on the non-negative cone of its domain $L_M^*$ with respect to the *-weak topology $\sigma(L_M^*,E_N)$. Moreover, the concavity of differential entropy is a standard result. Its proof can be found in \cite{Cover2006}. And we can easily derive that $-\mathcal{H}(\pi)$ is convex respect to $\pi$.\\
Furthermore, let us introduce some notations
$$E= \begin{pmatrix}
    1\\V'(x)\\ 
    \frac{1}{2}V''(x)
\end{pmatrix},\quad \widetilde{A}_\theta(t,x,u)=\begin{pmatrix}
    f_\theta\\ b_\theta \\
    \sigma^2_\theta
\end{pmatrix}(t,x,u).$$
From assumption (H2.3), it yields that
\begin{equation}
\begin{aligned}
    &|g_\theta^\pi(0)-g_{\theta'}^{\pi}(0)| = \int_\Theta \int_U \left\langle(\widetilde{A}_\theta(0,x,u)-\widetilde{A}_{\theta'}(0,x,u))\pi_0(u), E \right\rangle du\ Q(d\theta)\\ &
    \leq \int_\Theta\int_U (1+V'(x)+\frac12 V''(x))(\bar{\omega}_N(\mu(\theta,\theta') )+ |u|I_{\{|u|\geq N\}})\pi_0(u) \ du\ Q(d\theta)\\ &\leq C(V'(x), V''(x), x,\pi_0)(\bar{\omega}_N(\mu(\theta,\theta') )+N^{-1}) \quad \forall N>0.
\end{aligned}
\end{equation}
Given that $\Theta$ is local compact, we can derive that $\theta \rightarrow \mathbb{E}[\widetilde{L}_\theta(0,x,{\pi}_0) + V'(x)\widetilde{b}_\theta(0,x,{\pi}_0)+\frac12 V''(x)\widetilde{\sigma}_\theta^2(0,x,{\pi}_0)]$ is bounded and continuous. Since $\mathcal{Q}$ is compact, we can choosing a subsequence $\varepsilon_n\rightarrow0$ such that $\mathbf{Q}^{\varepsilon_n}$ converges weakly to $\mathbf{Q}^*$. It follows that 
\begin{equation}
\begin{aligned}
&\lim_{N\rightarrow\infty}\int_{\Theta}\widetilde{L}_\theta(0,x,{\pi}_0) + V'(x)\widetilde{b}_\theta(0,x,{\pi}_0)+\frac12 V''(x)\widetilde{\sigma}_\theta^2(0,x,{\pi}_0)\ Q^N(d\theta)\\ &=\int_{\Theta}\widetilde{L}_\theta(0,x,{\pi}_0) + V'(x)\widetilde{b}_\theta(0,x,{\pi}_0)+ \frac12 V''(x)\widetilde{\sigma}_\theta^2(0,x,{\pi}_0) Q^*(d\theta).
\end{aligned}    
\end{equation}
Then, according to Sion's minimax theorem, we can establishes the desired result.
\end{proof}
For an arbitrary measure $Q$, we introduce the value function for a fixed coefficient $\theta\in Q$
\begin{equation}
\begin{aligned}
    v_\theta(x) := \inf_{\pi\in \mathcal{A}(x,\theta)} y_\theta^\pi(0),\label{vtheta}
\end{aligned}    
\end{equation}
which is the solution of the corresponding HJB equation
\begin{equation}
    \begin{aligned}
        &\rho v_\theta(x)=\inf_{\pi \in \mathcal{A}(x,\theta)}  \widetilde{L}_\theta(0,x,{\pi}_0) + v_\theta'(x)\widetilde{b}_\theta(0,x,{\pi}_0)+\frac12 v_\theta''(x)\widetilde{\sigma}_\theta^2(0,x,{\pi}_0
        )\label{hjbtheta}.
    \end{aligned}
\end{equation}\\
\eqref{vtheta} should satisfy $$V(x)= \sup_{\mathbf{Q}\in\mathcal{Q}} \int_{\Theta}v_\theta(x)\ \mathbf{Q}(d\theta).$$
Proceeding with standard arguments, we deduce that the right hand side of \eqref{hjbtheta} is minimized at
\begin{equation}
\begin{aligned}
    &\pi_{\theta}^*(u,x) =\frac{exp\{f_{\theta}(x,u)+v_\theta'(x)b_{\theta}(x,u)+\frac{1}{2}v_\theta''(x) \sigma^2(x,u)\}}{\int_U exp\{f_{\theta}(x,u)+v_\theta'(x)b_{\theta}(x,u)+\frac{1}{2}v_\theta''(x)\sigma^2(x,u)\} du},
\end{aligned}   
\end{equation}
For each given initial state $x\in \mathbb{R}$, we can utilize the feedback control to generate an optimal open-loop control
\begin{equation}
\begin{aligned}
    &\pi^{\theta,*}_t =\frac{exp\{f_{\theta}(X_\theta^*(t),u)+v_\theta'(X_\theta^*(t))b_{\theta}(X_\theta^*(t),u)+\frac{1}{2}v_\theta''(X_\theta^*(t))\sigma^2(X_\theta^*(t),u)\}}{\int_U exp\{f_{\theta}(X_\theta^*(t),u)+v_\theta'(X_\theta^*(t))b_{\theta}(X_\theta^*(t),u)+\frac{1}{2}v_\theta''(X_\theta^*(t))\sigma^2(X_\theta^*(t),u)\} du.}
\end{aligned}   
\end{equation}
Note that the Sion's minimax theorem is crucial in above derivation. Accordingly, we can derive the optimal strategy distribution $\pi_\theta^*$ before determining the worst-case coefficient distribution $Q$. In order to use it, we assume that $b_\theta, \sigma_\theta$ and $f_\theta$ belongs to $\mathcal{G}^p(0,\infty; \mathbb{R})$ to ensure the integrand of \eqref{hjbthe} is continuous in $\pi$.

\section{The Linear-Quadratic Cases}

We now focus on the following linear quadratic stochastic control problems under model uncertainty. The problem is characterized by linear state dynamics and quadratic rewards, in which
\begin{equation}
b_\theta(x,u)=A_{\theta}x+B_{\theta}u, \ \ \sigma_\theta(x,u)= C_{\theta}x +D_{\theta}u,\ \ x,u \in \mathbb{R},\label{LQx}
\end{equation} 
where $A_\theta, B_\theta,C_\theta,D_\theta$ are continuous functions only depend on $\theta$. And
\begin{equation}
 f_{\theta}(x,u) =  \frac{L_{\theta}}{2}x^2+S_{\theta}xu + \frac{R_{\theta}}{2}u^2 + M_{\theta}x+ N_{\theta}u, \ x,u\in \mathbb{R},\label{LQf}  
\end{equation}
where $L_{\theta},S_{\theta},R_{\theta}, M_{\theta}, N_{\theta}$ are continuous functions only depends on $\theta$ and $L_\theta, R_\theta\gg0$. We assume that the control set is unconstrained, namely, $U=\mathbb{R}.$\\
Fix an initial state $|x|<M \ (M>0)$. Following the derivation in the previous section, we extend the state dynamics to the following relations
\begin{equation}
\begin{cases}
    dX^{\pi}_{\theta}(t) = \int_\mathbb{R}(A_\theta X^{\pi}_{\theta}(t)+B_\theta u){\pi}_t (u)du\ dt\\ \ \ \ \ \ \ \ \ \ \ \ \  +\sqrt{\int_\mathbb{R}(C_\theta X^{\pi}_{\theta}(t)+D_\theta u)^2{\pi}_t(u)\ du }\ dW(t), \quad t\in[0,\infty); \\
    X^{\pi}_{\theta}(0) = x.\quad x\in [-M,M]\ (M>0)\label{LQexpsde}
\end{cases}
\end{equation} \par
Next we specify the associated set of the admissible controls $\mathcal{A}(x,\theta)$ : for for each $Q \in\mathcal{Q} $ and $\theta \sim Q$, ${\pi}_t \in \mathcal{A}(x,\theta)$,if\par
(i) for each $t\geq0,\pi_t \in \mathcal{P}(\mathbb{R})$ a.s;\par
(ii)for each $A\in \mathcal{B}(\mathbb{R})$ ,$ \{\int_A \pi^{\theta}_t(u) du,t\geq0\}$ is $\mathcal{F}_t$-progressively measurable;\par
(iii)for each $t\geq0,\mathbb{E}[\int_0^t(|E^{\pi_\theta}[u]|^2+|Var
^{\pi_\theta}[u]|^2)ds]<\infty$, where $E^{\pi_\theta}[u]=\int_R u{\pi}_t(u)du,\ \  Var
^{\pi_\theta}[u]= \int_R u^2{\pi}_t(u)du- \left(E^{\pi_\theta}[u]\right)^2$;\par
(iv)For any $\theta\in \Theta $, with $\{X^{\pi}_{\theta}(t), t\geq 0\}$ solving \eqref{LQexpsde}, $\liminf_{T\rightarrow \infty}e^{-\rho T}\mathbb{E}[(X_{\theta}^{\pi}(T))^2]=0$;\par
(v)For any $\theta\in \Theta$, with $\{X^{\pi}_{\theta}(t), t\geq 0\}$ solving \eqref{LQexpsde}, $|y_\theta^\pi(0)|<M ,\ (M>0) $.\par
Similar to the classical case, we give the definitions of the value function. The exploratory value function of LQ problems under model uncertainty is given by
\begin{equation}
\begin{aligned}
 &V(X_\theta^\pi(t)) =\sup_{Q\in\mathcal{Q}}\inf_{\pi\in \mathcal{A}(x,\theta)}\mathbb{E}\left[\int_{\Theta}\int_0^{t} e^{-\rho t}(\int_\mathbb{R} f_\theta(X_\theta^\pi(t),u){\pi}_t(u) du - \alpha\mathcal{H}({\pi}_t) )\ dt \ Q(d\theta) \right].
\end{aligned}
\end{equation}
In the following sections, we derive explicit solutions for two cases: when the unknown parameter follows a two-point distribution and when it follows a uniform distribution.

\subsection{The case of two-point distribution coefficients}
In this section we assume that the Brownian motion is one dimension for simplicity.  $\Theta = \{1,2\}$ is a discrete space and
$$\mathcal{Q} = \{\mathbf{Q}^{\lambda}:\lambda\in [0,1]\},$$
where $\mathcal{Q}$ denotes a family of distributions for $\theta$ such that $\mathcal{Q}^{\lambda}(\{1\})= \lambda$ and $\mathcal{Q}^{\lambda}(\{2\})= 1-\lambda$.\par
Considering the case of the worst market condition, the cost function reduces to
\begin{equation}
\begin{aligned}
J(u) &= \sup_{Q^{\lambda
}\in\mathcal{Q}}\int_{\Theta}y_{\theta}^{\pi}(0) \ Q^{\lambda
}(d\theta)  = \max(y_1^\pi(0),y_2^\pi(0)).
\end{aligned}
\end{equation}
Let $\pi^*$ be an optimal control, the cost function can be reorganized as $\max(y_1^*(0),y_2^*(0))=\lambda y_1^*(0)+(1-\lambda)y_2^*(0)$ 
where $\lambda$ only takes value on 0 or 1.

\textbf{Assumption 4.1}
    (i)The discount rate satisfies $\rho > max\{2(\lambda A_1+\frac{\lambda}{2}C_1^2+\max[\frac{D_1^2S_1^2-2R_1S_1(B_1+C_1D_1)}{R_1},0]),$ $ 2((1-\lambda) A_2+\frac{1-\lambda}{2}C_2^2)+\max[\frac{D_2^2S_2^2-2R_2S_2(B_2+C_2D_2)}{R_2},0]\}$;\par
    (ii)With $V_1 = \frac{S_1S_2(C_1D_1B_2- C_2D_2B_1)}{(B_1B_2+C_2D_2B_1+C_1C_2D_1D_2)(S_2B_1+C_1D_1S_2)}$, \\ $V_2 = \frac{S_1S_2(C_2D_2B_1- C_1D_1B_2)}{(B_1B_2+C_1D_1B_2+C_1C_2D_1D_2)(S_1B_2+C_2D_2S_1)}$, assume that $R_\theta +D_\theta^2 V_\theta =0$,  $\theta = 1,2$,\par(iii)$(B_1B_2+C_1D_1B_2+C_1C_2D_1D_2)(S_1B_2+C_2D_2S_1)\neq0$ and $(B_1B_2+C_2D_2B_1+C_1C_2D_1D_2)(S_2B_1+C_1D_1S_2)\neq 0$,\par
    (iv)$S_1^2+D_1^2L_1V_1<0$ and $S_2^2 + D_2^2L_2V_2 <0$.

The exploratory value function with environment coefficient following a two point distribution is given by
\begin{equation}
\begin{aligned}
V(x) &=  \sup_{\lambda\in\{0,1\}}\inf_{\pi\in\mathcal{A}(x,\lambda)}\left[
 \lambda y_1^{\pi}(0) + (1-\lambda) y_2^{\pi}(0)\right]\label{LQ1V(x)}= \max\{V_1(x),V_2(x)\}.
\end{aligned}
\end{equation}
For convenience, set
$$
A = \begin{bmatrix}
    A_1 & \\& A_2
\end{bmatrix},
B = \begin{bmatrix}
    B_1 \\ B_2
\end{bmatrix},
C = \begin{bmatrix}
    C_1 & \\ &C_2
\end{bmatrix},
D = \begin{bmatrix}
    D_1 \\ D_2
\end{bmatrix},
L = \begin{bmatrix}
    L_1 & \\ &L_2
\end{bmatrix},$$$$
S = \begin{bmatrix}
    S_1 \\ S_2
\end{bmatrix},
M = \begin{bmatrix}
    M_1 \\ M_2
\end{bmatrix},
v'(x) =\begin{bmatrix}
    \frac{\partial v(x)}{\partial x_1} \\ \frac{\partial v(x)}{\partial x_2}
\end{bmatrix},
v''(x) =\begin{bmatrix}
    \frac{\partial^2 v(x)}{\partial x_1^2}  & \frac{\partial^2 v(x)}{\partial x_1\partial x_2}\\ \frac{\partial^2 v(x)}{\partial x_2^2}  & \frac{\partial^2 v(x)}{\partial x_2\partial x_1}
\end{bmatrix},$$$$
x = \begin{bmatrix}
    x_1 \\ x_2
\end{bmatrix},
b(t,x) = \begin{bmatrix}
    b_1(t,x_1) \\ b_2(t,x_2)
\end{bmatrix},
\sigma(t,x) = \begin{bmatrix}
    \sigma_1(t,x_1) \\ \sigma_2(t,x_2)
\end{bmatrix},
\Lambda = \begin{bmatrix}
    \lambda & \\ &1-\lambda
\end{bmatrix},
$$
And $\widetilde{R}=\lambda R_1+(1-\lambda)R_2$, $\widetilde{N}=\lambda N_1+(1-\lambda)N_2$,\\ $U_1 = \frac{(B_1B_2+C_1D_1B_2+C_1C_2D_1D_2)(S_1B_2+C_2D_2S_1)}{(B_1B_2+C_2D_2B_1+C_1C_2D_1D_2)(S_2B_1+C_1D_1S_2)}$, \\$U_2 = \frac{(B_1B_2+C_2D_2B_1+C_1C_2D_1D_2)(S_2B_1+C_1D_1S_2)}{(B_1B_2+C_1D_1B_2+C_1C_2D_1D_2)(S_1B_2+C_2D_2S_1)}$.  

The integration reaches maximum at $\lambda=1$ or $\lambda=0$. Therefore $x_1$ and $x_2$ can not simultaneously exist \eqref{LQ1V(x)}, from which we deduce that $\frac{\partial^2 v }{\partial x_2\partial x_1} = \frac{\partial^2 v }{\partial x_1\partial x_2}=0$. With standard arguments in \cite{yong1999stochastic}, We deduce that the exploratory value function is a solution of the following exploratory HJB equation
\begin{equation}
\begin{aligned}
    &\rho V(x)= \max_{\lambda\in\{0,1\}} \inf_{\pi\in\mathcal{A}(x,\lambda)}\int_\mathbb{R}\frac12 x^{\top}\Lambda L x+S^{\top}\Lambda xu+\frac12\widetilde{R}u^2+M^\top\Lambda x+\widetilde{N}u+ \\ &(V'(x))^\top\Lambda A x+ B^\top\Lambda V'(x)u+\frac12(x^\top C\Lambda C V''(x)x+2D^\top C\Lambda V''(x)xu+D\Lambda Du^2)\\& +\alpha \ln\pi^\lambda(u)]\pi^\lambda(u)\ du.\label{LQ1hjb}
\end{aligned}
\end{equation}

For any $x\in \mathbb{R}^n$, the right-hand side of\eqref{LQ1hjb} can be solved to obtain a candidate optimal feedback control as
\begin{small}
\begin{equation}
\begin{aligned}
    \pi^*(u;x) &= 
     \mathcal{N}(u|-\frac{S^{\top}\Lambda x+\widetilde{N}+ B^{\top}\Lambda V'(x)+ D^{\top}C\Lambda V''(x)x}{\widetilde{R}+D^{\top}V''(x)\Lambda D},\frac{\alpha}{\widetilde{R}+D^{\top}V''(x)\Lambda D} )\\ &=\mathcal{N}(\mu^\pi,(\sigma^\pi)^2)\label{LQ1pi}
\end{aligned}    
\end{equation}
\end{small}
Therefore, the optimal feedback control distribution appears to be Gaussian with mean $\mu^{\pi}$ and variance $(\sigma^{\pi})^2$.\par Substituting the candidate optimal feedback control policy \eqref{LQ1pi} back into \eqref{LQ1hjb}, then the HJB equation is reduced to
\begin{equation}
\begin{aligned}
    &\rho V(x) = \max_{\lambda}\{\frac12x^{\top}\Lambda Lx + M^{\top}\Lambda x + V'(x)^{\top}A\Lambda x + \frac12x^{\top}C\Lambda CV''(x) x \\ &-\frac{\alpha}{2} \ln\left(\frac{e^2\pi  \alpha }{\widetilde{R} + D^{\top}\Lambda V''(x)D}\right) - \frac{\left(S^{\top}\Lambda x + \widetilde{N} + B^{\top}\Lambda V'(x)+ D^{\top}C\Lambda V''(x)x\right)^2}{2(\widetilde{R}+D^{\top}\Lambda V''(x)D)}  \}  .\label{LQ1hjbsub}
\end{aligned}    
\end{equation}
 Under Assumption 4.1, one smooth solution to the HJB equation\eqref{LQ1hjbsub} is given by
$V(x)= \frac12x^\top K_2 x+K_1^\top x+K_0,$where
\begin{scriptsize}
\begin{equation}
\begin{aligned}
    &K_2=\begin{bmatrix}
        \frac{\lambda^{1/2}\left(G_1-\sqrt{G_1^2-4F_1H_1}\right)}{2F_1} & \\ & \frac{(1-\lambda)^{1/2}\left(G_2-\sqrt{G_2^2-4F_2H_2}\right)}{2F_2}
    \end{bmatrix}=\begin{bmatrix}
        k_{21} & \\ & k_{22}
    \end{bmatrix},\\&
    K_1 = \begin{bmatrix}
        \frac{\lambda P_1 }{O_1}  ,\ \  \frac{(1-\lambda) P_2 }{O_2}
        \end{bmatrix}^\top =[k_{11},\ k_{12}],\\ & 
    K_0= \frac{-1}{2\left(\widetilde{R}+D^\top \Lambda V''(x)D\right)\rho}\left[\widetilde{N}^2  + 2\widetilde{N}B^\top \Lambda K_1+  (B^\top \Lambda K_1)^2\right]- \frac{\alpha}{2}\ln(\frac{e^2\pi\alpha}{\widetilde{R}+D^\top \Lambda V''(x)D}).\label{LQ1Ks}
\end{aligned}    
\end{equation}
\end{scriptsize}
In above equations, 
$F_1 =[2\lambda D_1^2(\lambda A_1-\frac{\rho}{2})-\lambda^2 B_1^2 -2\lambda^2 C_1D_1B_1 +2(1-\lambda)D_2^2(\lambda A_1+\frac{\lambda}{2}C_1^2-\frac{\rho}{2})U_2]$; \\ $ G_1 = \lambda\left[\lambda D_1^2 L_1  - 2 \lambda S_1 B_1 - 2\lambda C_1 D_1 S_1 + 4 R_1(\lambda A_1 +\frac{\lambda}{2}C^2_1-\frac{\rho}{2})\right]$; \\ $H_1 = -\lambda^2 (S_1^2 + D_1^2 L_1 V_1)$; $O_1= 2(\widetilde{R}+D^\top \Lambda V''(x)D)(\rho -\lambda A_1) + 2\lambda^2 B_1 (S_1 + k_{21} + C_1 D_1 k_{21})$;  $P_1 = 2(\widetilde{R}+D^\top \Lambda K_2 D)M_1 -2 S_1\widetilde{N}-2N_1[\lambda (B_1+C_1D-1)k_{21} + (1-\lambda)(B_2+C_2D_2) k_{22}]$. The expression of $F_2$, $G_2$, $H_2$, $O_2$, $P_2$ have symmetric structures with their counterparts, thus we omit the explicit definitions.\par
Now We derive explicit solution of \eqref{LQ1hjbsub}. Note this value function still depends on environment coefficient $\lambda$. Given the specific forms of state dynamics\eqref{LQx} and cost function\eqref{LQf}, we can determine the value of $\lambda$ by taking maximum between $V_1(x)$ and $V_2(x)$.\par
We demonstrate the system of algebraic equations to derive \eqref{LQ1hjbsub} in Appendix.\par
We now provide the proof that the value function satisfies the conditions for the verification theorem to hold. \par
\begin{theorem}
\label{LQ1thm}
    Suppose Assumption 4.1 holds. Then for any $x_\theta \in\mathbb{R},\ \theta=1,2$, the exploratory LQ problem with two-point distribution coefficient has a value function $V(x)$ given by
    $$V(x)= \frac12x^\top K_2 x+K_1^\top x+K_0,$$
    satisfies the HJB equation\eqref{LQ1hjb}, where $K_2,\ K_1$ and $K_0$ are defined by \eqref{LQ1Ks} respectively.\par
    Moreover, the optimal feedback control is Guassian, with its density function given by
\begin{small}
\begin{equation}
\begin{aligned}
    \pi^*(u;x) &= 
    \mathcal{N}(u|-\frac{(S^\top\Lambda +B^\top\Lambda K_2 + D^\top C\Lambda K_2)x+\widetilde{N}+B^\top\Lambda K_1}{\widetilde{R}+D^{\top}K_2\Lambda D}  ,\frac{\alpha}{\widetilde{R}+D^{\top}K_2\Lambda D})\\ &=\mathcal{N}(\mu^*,(\sigma^*)^2)\label{LQ1pistar}
\end{aligned}    
\end{equation}
\end{small}
Finally, the associated optimal state process $\{X_t^{\pi^*},\ t\geq0\}$ under $\pi^*$ is the unique solution of the SDE
\begin{equation}
\begin{cases}
    dX^{\pi}_{\theta}(t) = (A_\theta X^{\pi}_{\theta}(t)+B_\theta \mu_t^*)\ dt\\ \ \ \ \ \ \ \ \ \ \ \ \ +\sqrt{(C_\theta X_\theta^\pi(t)+D_\theta\mu_t^*)^2+D_\theta^2 (\sigma_t^*)^2 }\ dW(t), \quad t\in[0,\infty); \\
    X^{\pi}_{\theta}(0) = x\in \mathbb{R}. x\in [-M,M]\ (M>0)\label{LQexpsde2}
\end{cases}
\end{equation}
\end{theorem}
\begin{proof}
    We fix the initial state $x\in \mathbb{R}$. Let $\pi \in \mathcal{P}(x)$ and $X_{\theta}^{\pi},\theta=1,2$ be the associated state process solving \eqref{LQexpsde2} with $\pi$ being used. Let $T>0$ be arbitrary. Define the stopping times $\tau_n^{\pi}:= \{t\geq 0: \int_0^t (e^{-\rho t} V'(X_{\theta,t}^{\pi})\tilde{\sigma}(X_{\theta,t}^{\pi},\pi_t))^2 dt \geq n\}$, for $n\geq 1$.Then, Ito's lemma yields
\begin{equation}
\begin{aligned}
    &e^{-\rho (T \land \tau_n^{\pi})}V(X_{T \land \tau_n^{\pi}}^{\pi}) =V(x) + \int_0^{T\land \tau_n^{\pi}}
    e^{-\rho t}(-\rho V(X_{t}^{\pi}) + V'(X_{T \land \tau_n^{\pi}}^{\pi})^{\top}\Lambda\tilde{b}(X_{T \land \tau_n^{\pi}}^{\pi},\pi_t) + \\ &\frac12\tilde{\sigma}(X_{T \land \tau_n^{\pi}}^{\pi},\pi_t)^{\top}V''(X_{T \land \tau_n^{\pi}}^{\pi})\Lambda \tilde{\sigma}(X_{T \land \tau_n^{\pi}}^{\pi},\pi_t)dt + \int_0^{T\land \tau_n^{\pi}} e^{-\rho t}
    V'(X_{T \land \tau_n^{\pi}}^{\pi})^{\top}\Lambda \tilde{\sigma}(X_{T \land \tau_n^{\pi}}^{\pi},\pi_t)dW_t
\end{aligned}    
\end{equation}
Taking expectations, using that $V(x)$ solves the HJB equation\eqref{LQ1hjb} and that $\pi$ is in general suboptimal yield
\begin{equation}
\begin{aligned}
    &\mathbb{E} [e^{-\rho (T \land \tau_n^{\pi})}V(X_{T \land \tau_n^{\pi}}^{\pi})]  =V(x) + \mathbb{E}[\int_0^{T\land \tau_n^{\pi}}
    e^{-\rho t}(-\rho V(X_{T \land \tau_n^{\pi}}^{\pi}) + V'(X_{T \land \tau_n^{\pi}}^{\pi})^{\top}\Lambda\tilde{b}(X_{T \land \tau_n^{\pi}}^{\pi},\pi_t) \\ & + \frac12\tilde{\sigma}(X_{T \land \tau_n^{\pi}}^{\pi},\pi_t)^{\top}V''(X_{T \land \tau_n^{\pi}}^{\pi})\Lambda \tilde{\sigma}(X_{T \land \tau_n^{\pi}}^{\pi},\pi_t)dt] \geq \\&V(x) -
    \mathbb{E}[\int_0^{T\land \tau_n^{\pi}} e^{-\rho t}(\lambda \widetilde{f}_1(t,X_t^{\pi},\pi) + (1-\lambda)\widetilde{f}_2(t,X_t^{\pi},\pi)  +   \alpha\int_U\pi(u)\ln\pi(u)du) dt].
\end{aligned}    
\end{equation}
According to Burkholder-Davis-Gundy inequality, for a semi-martingale we always have
\begin{equation}
\begin{aligned}
& \mathbb{E}[\sup_{0\leq t\leq T}|X_{\theta,t}^{\pi}|^2] \leq \mathbb{E}[|X_{\theta,T}^{\pi}|^2] \\ &= x^2 + \mathbb{E}\left[\int_0^T(B_{\theta}E^\pi[u]+  2C_{\theta}D_{\theta}E^\pi[u])|X_{\theta,t}^{\pi}|+D_{\theta}^2(E^\pi[u]^2 + Var^\pi[u]^2) dt\right] \\ &+ \int_0^T(2A_{\theta}+  C_{\theta}^2)\mathbb{E}|X_{\theta,T}^{\pi}|^2 dt.
\end{aligned}    
\end{equation}
The last step can be deduced from Ito lemma. On the other hand, applying Gronwall's inequality, we can certainly have
$$
\mathbb{E}[\sup_{0\leq t\leq T}|X_{\theta,t}^{\pi}|^2]\leq (x_{\theta}^2+f_1^\theta)e^{(2A_{\theta}+C_{\theta}^2)T}
$$
where $f_1^\theta $ is a function constant independent of n. Considering $\Theta$ is compact and $f_1^\theta$ is continuous of $\theta$, we can deduce that $f_1^\theta$ is a bounded function. By analogous arguments, we can also have $\mathbb{E}[|X_{\theta,t}^{\pi}|]\leq (|x|+\int_0^TB_{\theta}udt) e^{A_{\theta}T}$. Sending $n\rightarrow\infty$, with the fact that $|f(x)| I_{\{0,T\land \tau_n^{\pi}\}}\leq |f(x)| I_{\{0,T\}}$ we deduce that
\begin{equation}
\begin{aligned}
    &\mathbb{E} [e^{-\rho T }V(X_{T }^{\pi})] \geq V(x) -\\ &
    \mathbb{E}\left[\int_0^{T} e^{-\rho t}(\lambda \widetilde{f}_1(X_{\theta,t}^{\pi},\pi_t) + (1-\lambda)\widetilde{f}_2(X_{\theta,t}^{\pi},\pi_t)  + \alpha\int_U\pi(u)\ln\pi(u)du) dt\right].
\end{aligned}    
\end{equation}
where we have used the dominated convergence theorem. Next, we recall the admissible condition \\ $\lim \inf_{T\rightarrow\infty}\mathbb{E}[e^{-\rho t}|X_{\theta,t}^{\pi}|^2]=0$ and $\rho > 2A_{\theta}+C_{\theta}^2$. This condition and the fact that $a>0$ lead to $\lim\inf_{T\rightarrow\infty}\mathbb{E}[e^{-\rho T}v(X_{T}^{\pi})]=0.$ According to the dominated convergence theorem we can deduce that
\begin{equation}
\begin{aligned}
    &V(x)\leq \mathbb{E}\left[\int_0^{\infty} e^{-\rho t}(\lambda \widetilde{f}_1(X_{\theta,t}^{\pi},\pi_t) +  (1-\lambda)\widetilde{f}_2(X_{\theta,t}^{\pi},\pi_t)  + \alpha\int_U\pi_t(u)\ln\pi_t(u)du) dt\right].
\end{aligned}    
\end{equation}
for each $x\in \mathbb{R}$ and $\pi \in \mathcal{A}(x)$. Hence, $v(x)\geq V(x)$, for all $x\in \mathbb{R}.$\par
On the other hand, let $\pi^*=\{\pi_t^*,t\geq0\}$ be the open-loop control distribution generated from the above feedback law along with the corresponding state process $\{X_{\theta,t}^*,t\geq 0\}$ with $X_{\theta,0}^*=x_{\theta}$, and assume that $\pi^* \in \mathcal{A}(x,\theta)$ for now. Then we have
\begin{equation}
\begin{aligned}
    &\mathbb{E} [e^{-\rho T }V(X_{T }^{*})] = V(x) -
    \mathbb{E}[\int_0^{T} e^{-\rho t}(\lambda \widetilde{f}_1(X_{\theta,t}^{*},\pi_t^*) +  (1-\lambda)\widetilde{f}_2(X_{\theta,t}^{*},\pi_t^*)  \\ &+ \alpha\int_U\pi_t^*(u)\ln\pi_t^*(u)du) dt].
\end{aligned}    
\end{equation}
With the fact that $\lim\sup_{T\rightarrow \infty}\mathbb{E}[e^{-\rho T}V(X_{\theta,T}^*)]\geq \lim\inf_{T\rightarrow \infty}\mathbb{E}[e^{-\rho T}V(X_{\theta,T}^*)]=0 $, and applying the dominated convergence theorem \eqref{w} yields
\begin{equation}
\begin{aligned}
     V(x) \geq
    \mathbb{E}[\int_0^{\infty} e^{-\rho t}(\lambda \widetilde{f}_1(X_{\theta,t}^{*},\pi_t^*) + (1-\lambda)\widetilde{f}_2(X_{\theta,t}^{*},\pi_t^*)  + \alpha\int_U\pi_t^*(u)\ln\pi_t^*(u)du) dt].\label{w}
\end{aligned}    
\end{equation}
for each $x \in \mathbb{R}$. This proves that $V$ is indeed the value function.\par 
It remains to show that $\pi^*\in \mathcal{A}(x)$. The proof of $\liminf_{T\rightarrow\infty}e^{-\rho T }\mathbb{E}[(X_\theta^*(T))^2]=0$ is very similar to that of Theorem 4 in \cite{wang2020reinforcement}, and is thus omitted.\par
Now, we establish the admissibility constraint
$$|y_\theta^*(0)|\leq M, (M>0).$$
Given the form of $f(x,u)$, we obtain that
\begin{equation}
    \begin{aligned}
        &|y^*_\theta(0)| \leq \mathbb{E}\left[\int_0^\infty e^{-\rho t}\left| \widetilde{f}_\theta(X^\pi_\theta(t),\pi_t(u))- \alpha\mathcal{H}(\pi_t) \right|dt\right]\leq \\ &\mathbb{E}[\int_0^\infty e^{-\rho t}| \int_U \left(\frac{L_\theta}{2}(X_\theta^*(t))^2+S_\theta X_\theta^*(t)u +\frac{R_\theta}{2}u^2 +M_\theta X_\theta^*(t)+N_\theta u\right)\pi_t^*(u)\ du\\ &- \frac{\alpha}{2}\ln(\frac{2\pi e \alpha}{R_\theta+D_\theta^2k_2}) |dt].
    \end{aligned}
\end{equation} 
Recall that
\begin{small}
    \begin{equation}
\begin{aligned}
    &\pi_t^*(u;x) = \\ &
    \mathcal{N}(u|-\frac{(S^\top\Lambda +B^\top\Lambda K_2 + D^\top C\Lambda K_2)X_\theta^*(t)+\widetilde{N}+B^\top\Lambda K_1}{\widetilde{R}+D^{\top}K_2\Lambda D}  ,\frac{\alpha}{\widetilde{R}+D^{\top}K_2\Lambda D}),
\end{aligned}    
\end{equation}
\end{small}
for any $t\geq 0$.
The moment exponential stability of $\mathbb{E}|X^*_\theta(t)|^2$ yields that $$\mathbb{E}\left[\int_0^\infty e^{-\rho t}(X_\theta^*(t)^2)\ dt\right]<\mathbb{E}\left[\int_0^\infty \frac12 x^2e^{-(\rho+\beta) t} \ dt\right]=\frac{1}{2(\rho+\beta)}.$$ Then we can derive that $|y_\theta^*(0)|\leq C(x, \alpha, \rho,\beta, \theta )$. 
\end{proof}
\subsection{The case of continuous distribution coefficient}
We now consider the case with the continuous distribution market coefficient. Suppose that $$\theta \sim U(0,a) \quad a\in [a_1,a_2].$$
where $a_1,a_2$ are two constants in $\mathbb{R}$ and $a_1<a_2$.

In this case, the cost function is given by
\begin{equation}
\begin{aligned}
    &V(x) =\sup_{a}\inf_{\pi\in \mathcal{A}^{(x)}} \int_0^a\frac1a E [ \int_0^{\infty} e^{-\rho t}(\widetilde{f}_{\theta}(X_{\theta}^{\pi}(t),{\pi}_t(u)) \\ &+\alpha \int_U {\pi}_t(u)\ln{\pi}_t(u) du ) dt| x_{\theta}^{\pi}(0) ] d\theta.\label{LQ2vf}
\end{aligned}
\end{equation}
where
\begin{equation}
\begin{aligned}
&\widetilde{f}_{\theta}(X^{\pi}_{\theta}(t),\pi_t(u)) = \int_{\mathbb{R}}\frac12[L_{\theta}x_{\theta}^2(s) +2S_{\theta}x_{\theta}(s)u +R_{\theta}u^2 + 2M_{\theta}x_{\theta}(s) + 2N_{\theta}u]\pi_t(u) du.
\end{aligned}
\end{equation}
In parallel, we also introduce $v_\theta(x)$ as \eqref{vtheta} defined.
We will be working with the following assumption.

\textbf{Assumption 4.2}for each $a$, $\theta\sim U(0,a)$, we have $  \rho > sup_{\theta}\{2A_{\theta}+C^2_{\theta}+\max[\frac{D_{\theta}^2S_{\theta}^2-2R_{\theta}S_{\theta}(B_{\theta}+C_{\theta}D_{\theta})}{R_{\theta}},0]\}$\label{assumption2} \label{LQ2assumption}

Proceeding with standard arguments, we deduce that V satisfies the Hamilton-Jacobi-Bellman(HJB) equation
\begin{equation}
\begin{aligned}
    \rho v_{\theta}(x_{\theta})&= \inf_{\pi\in \mathcal{P}(U)}\int_R [\frac12 L_{\theta}x_{\theta}^2+S_{\theta}x_{\theta}u +\frac12R_{\theta}u^2+ M_{\theta}x_{\theta}+N_{\theta}u+v'_{\theta}(x_{\theta})(A_{\theta}x_{\theta}+B_{\theta}u)\\ &+\frac12 v''_{\theta}(x_{\theta})(C_{\theta}x_{\theta}+D_{\theta}u)^2 +\alpha \ln \pi_{\theta}(u)]\pi_{\theta}(u)du \label{hjb2}
\end{aligned}    
\end{equation}
Following an analogous argument as for \eqref{LQ1pi}, we deduce that a candidate optimal feedback control is given by

\begin{equation}
\begin{aligned}
    \pi^*_{\theta}(u;x) = \mathcal{N}&(u|-\frac{S_{\theta} x_{\theta}+N_{\theta}+ B_{\theta}v'_{\theta}(x_{\theta})+C_{\theta}D_{\theta}x_{\theta}v''_{\theta}(x_{\theta})}{R_{\theta}+D_{\theta}^2v_{\theta}''(x_{\theta})},\frac{\alpha}{R_{\theta}+D_{\theta}^2v_{\theta}''(x_{\theta})} )
\end{aligned}    
\end{equation}
Letting $\mu_{\theta}^*(x_{\theta})$ and $(\sigma^*_{\theta}(x_{\theta}))^2$ denote the mean and variance of the optimal distribution $\pi^*_{\theta}(\cdot;x)$ as previously defined, the Hamilton–Jacobi–Bellman (HJB) equation \eqref{hjb2} can be equivalently reorganized as follow

\begin{equation}
    \begin{aligned}
        \rho v_{\theta}(x_{\theta}) &= -\frac{(S_{\theta}x_{\theta}+N_{\theta}+B_{\theta}v_{\theta}'(x_{\theta})+C_{\theta}D_{\theta}x_{\theta}v_{\theta}''(x_{\theta}))^2}{2(R_{\theta}+D_{\theta}v_{\theta}''(x_{\theta}))}+\frac12L_{\theta}x^2_{\theta}+M_{\theta}x_{\theta}\\&+v'_{\theta}(x_{\theta})A_{\theta}x_{\theta}+\frac12v_{\theta}''(x_{\theta})C_{\theta}^2x_{\theta}^2-\frac{\alpha}{2}(\ln(2\pi e\sigma^2_{\theta})-1).\label{reorg hjb}
    \end{aligned}
\end{equation} 
Under \textbf{Assumption \ref{LQ2assumption}} and the additional condition\\$R_{\theta}L_{\theta}>S^2_\theta$,a smooth solution to the HJB equation\eqref{reorg hjb} is given by
$$v_{\theta}(x_{\theta})=\frac12k_2^\theta x_{\theta}^2 +k^{\theta}_1x_{\theta}+k^{\theta}_0,$$
where
\begin{equation}
    \begin{aligned}
        &k^{\theta}_2=\frac{-\widetilde{b_{\theta}}-\sqrt{\widetilde{b_{\theta}}^2-4\widetilde{a_{\theta}}\widetilde{c_{\theta}}}}{2\widetilde{a_{\theta}}}
        \\ &k^{\theta}_1=\frac{(R_{\theta}+D_{\theta}^2k^\theta_2)M_{\theta}-N_{\theta}(S_{\theta}+B_{\theta}k_2^\theta+C_{\theta}D_{\theta}k_2^\theta)}{(R_{\theta}+D_{\theta}^2k_2^\theta))(\rho -A_{\theta})+B_{\theta}(S_{\theta}+B_{\theta}k_2^\theta+C_{\theta}D_{\theta}k_2^\theta)}
        \\&
        k^{\theta}_0=-\frac{N_{\theta}^2+2B_{\theta}N_{\theta}k_1^\theta+B_{\theta}^2(k_1^\theta)^2}{2\rho(R_{\theta}+D_{\theta}^2k_2^\theta)}-\frac{\alpha(\ln(2\pi e\sigma_{\theta}^2)+1)}{2\rho}
    \end{aligned}
\end{equation}
$\widetilde{a_{\theta}}=(C_{\theta}^2+2A_{\theta}-\rho)D_{\theta}^2-(B_{\theta}+C_{\theta}D_{\theta})^2$, $\widetilde{b_{\theta}}=(C_{\theta}^2+2A_{\theta}-\rho)R_{\theta}-2S_{\theta}(B_{\theta}+C_{\theta}D_{\theta})+D^2_{\theta}L_{\theta}$, $\widetilde{c_\theta}=R_{\theta}L_{\theta}-S_{\theta}^2.$
The value function takes the following form 
$$V(x) = \sup_a\int_0^a \frac1a (\frac12 k^{\theta}_2x_{\theta}^2+k^{\theta}_1x_{\theta}+k^{\theta}_0)d\theta.$$
We now provide the solvability equivalence between the classical problem and exploratory problem.With analogous proof in \cite{wang2020reinforcement} in Theorem 7, we can deduce that
\begin{equation}
    \begin{aligned}
        v_\theta(x_\theta)\leq \mathbb{E}[\int_0^{\infty} e^{-\rho t}|L(X_{\theta}^{\pi}(t),\pi_t)|dt]
    \end{aligned}
\end{equation}
and
\begin{equation}
    \begin{aligned}
        v_\theta(x_\theta)\geq \mathbb{E}[\int_0^{\infty} e^{-\rho t}|L(X_{\theta}^{\pi^*}(t),\pi^*_t)|dt].
    \end{aligned}
\end{equation}
Integrate both sides of the equation with respect to $\theta$ we have
\begin{equation}
    \begin{aligned}
        \sup_a \int_0^a \frac1a v_\theta(x_\theta) d\theta\leq \sup_a\int_0^a \frac1a\mathbb{E}[\int_0^{\infty} e^{-\rho t}|L(X_{\theta}^{\pi}(t),\pi_t)|dt]d\theta
    \end{aligned}
\end{equation}
and
\begin{equation}
    \begin{aligned}
        \sup_a\int_0^a \frac1a v_\theta(x_\theta) d\theta\geq\sup_a \int_0^a \frac1a \mathbb{E}[\int_0^{\infty} e^{-\rho t}|L(X_{\theta}^{\pi^*}(t),\pi^*_t)|dt]d\theta.
    \end{aligned}
\end{equation}
This proves that $v$ is indeed the value function, namely 
\begin{equation}
\begin{aligned}
    &V(x) \equiv \sup_{a}\inf_{\pi\in \mathcal{A}^{(x)}} \mathbb{E}\left[\int_0^a  \int_0^{\infty} \frac{e^{-\rho t}}{a} (\widetilde{f}_{\theta}(X_{\theta}^{\pi}(t),\pi_t(u)) +\alpha \int_U \pi_t(u)\ln\pi_t(u) du ) dt \ d\theta\right] .
\end{aligned}
\end{equation}
Now we establish the proof of $\pi^* \in\mathcal{A}(x,\theta)$. Under Assumption 4.1, we obtain that $v_\theta''(x)$ is a positive constant, implying that the denominators in both the mean and variance of \eqref{LQ1pistar} are strictly positive. Given the continuity of the functions $A_\theta$, $B_\theta$, etc. with respect to $\theta$, it then follows that the optimal policy $\pi^*$is continuous in $\theta$.

\section{Sovability equivalence between classical and expolrtory problems}
To illustrate the solvability equivalence between the classical LQ problem and the exploratory LQ problem under model uncertainty, this section considers a setting in which the market coefficients follow a two-point distribution.\par
Firstly we recall the classical LQ
problem under model uncertain. With $\{W_t, t\geq0\}$ be a standard Brownian motion on filtered probability space $(\Omega, \mathcal{F}, \{\mathcal{F}_t\}_{t\geq0},\mathbb{P})$ satisfying the usual conditions. The state dynamic $\{x_\theta(t)^u\}$ solves
\begin{equation}
\begin{cases}
    dx^{u}_{\theta}(t) = (A_\theta x^{u}_{\theta}+B_\theta u)dt+(C_\theta x^{u}_{\theta}+D_\theta u)dW(t), \quad t\in[0,\infty); \\
    x^{u}_{\theta}(0) = x.\label{LQ2clsde}
\end{cases}
\end{equation}
The value function is defined as follow
\begin{equation}
    \begin{aligned}
        V_\theta^{cl}(x) = \inf_{\pi\in A(x,\theta)} \sup_{Q\in \mathcal{Q}}\int_\Theta\mathbb{E}\left[\int_0^\infty e^{-\rho t}f_\theta(x_\theta(t),u)\  dt\right]Q(d\theta),\label{LQ2clvf}
    \end{aligned}
\end{equation}
where $f_\theta(t,x_\theta(t),u)$ is given by \eqref{LQf}.
Next we provide the admissible control. For each $Q\in \mathcal{Q}$ and $\theta\sim Q$, set $\mathcal{A}(x, \theta) $: $u^\theta \in \mathcal{A}(x, \theta)$ if\par
(i) $\{u_t^\theta, t\geq0\}$ is $\mathcal{F}_t$-progressively measurable;\par
(ii)For each $t\geq 0, \mathbb{E}[\int_0^t(u_s^\theta)^2 ds]\leq \infty$;\par
(iii)For any $\theta \in \Theta$, with $\{x_\theta(t)^{u}, t\geq0\}$ solving \eqref{LQ2clsde}, $\liminf_{T\rightarrow\infty}e^{-\rho T}\mathbb{E}[(x_\theta^u(T))^2]=0$;\par
(iv)For any $\theta \in \Theta$, with $\{x_\theta^{u}(t), t\geq0\}$ solving \eqref{LQ2clsde}, $|y^{cl}_\theta(0)|< M,(M>0).$\par
We now provide the theorem of solvability equivalence between two problems.
\begin{theorem}
The following two statements (a) and (b) are equivalent.\par
\textbf{(a)} The function $V(x) = \sup_a\int_0^a \frac1a (\frac12 k^{\theta}_2x_{\theta}^2+k^{\theta}_1x_{\theta}+k^{\theta}_0)d\theta $, $ x\in \mathbb{R}$, with $k_1^\theta,k_0^\theta\in \mathbb{R}$ and $k^\theta_2>0$, is the value function of the exploratory problem \eqref{LQ2clvf} and the corresponding optimal feedback control is \label{eqthm}
\begin{equation}
\begin{aligned}
    \pi^*_{\theta}(u;x) = \mathcal{N}&(u|-\frac{ x_{\theta}(S_\theta+B_\theta k_2^\theta+C_{\theta}D_{\theta}k_2^\theta)+N_{\theta}+ B_{\theta}k_1^\theta}{R_{\theta}+D_{\theta}^2k_2^\theta},\frac{\alpha}{R_{\theta}+D_{\theta}^2k_2^\theta)} ).
\end{aligned}    
\end{equation}
\par
\textbf{(b)} The function $V(x)^{cl}= \sup_a\int_0^a \frac1a (\frac12 k^{\theta}_2x_{\theta}^2+k^{\theta}_1x_{\theta}+k^{\theta}_0+\frac{\alpha(\ln(2\pi e\sigma_{\theta}^2)+1)}{2\rho})d\theta $, $ x\in \mathbb{R}$, with $k_1^\theta,k_0^\theta\in \mathbb{R}$ and $k^\theta_2>0$, is the value function of the exploratory problem \eqref{LQ2vf} and the corresponding optimal feedback control is
\begin{equation}
\begin{aligned}
    u_\theta^*= -\frac{ x_{\theta}(S_\theta+B_\theta k_2^\theta+C_{\theta}D_{\theta}k_2^\theta)+N_{\theta}+ B_{\theta}k_1^\theta}{R_{\theta}+D_{\theta}^2k_2^\theta}.
\end{aligned}    
\end{equation}
\end{theorem}
 The proof of solvability equivalence between classical and exploratory problem is very similar to that of Theorem 7 in \cite{wang2020reinforcement} and is thus omitted.\par
To conclude this section, we examine the cost incurred by exploration, which is Originally formulated and derived by \cite{wang2020reinforcement} in the context of an infinite-horizon setting. 
\begin{equation}
\begin{aligned}
   & C^{u^*,\pi^*}(x):= V(x)+\mathbb{E}\left[\int_{a_1}^{a2}\frac{\alpha}{a}\int_0^\infty\int_{\mathbb{R}}\pi^*_{t,\theta}(u)\ln\pi^*_{t,\theta}(u) \ du\ d\theta\right]-V^{cl}(x).
\end{aligned}
\end{equation}
where $a$ is a constant maximizing cost function \eqref{Jpi}. 
\begin{theorem}
    Suppose that \textbf{Theorem \ref{eqthm}} holds. Then, the exploration cost for the uncertain model LQ problem is 
$$C^{u^*,\pi^*}(x)=-\frac{\alpha}{2\rho},\ \  for\ \ x\in \mathbb{R}.$$
\end{theorem}
\begin{proof}
  Let $\{\pi_{t,\theta}^*(u)\}$ be the open-loop control generated by the feedback control $\pi^*$ given in statements (a) with respect to the initial state $x$, namely,
\begin{equation}
\begin{aligned}
    \pi^*_{\theta}(u;x) = \mathcal{N}&(u|-\frac{ x_{\theta}(S_\theta+B_\theta k_2^\theta+C_{\theta}D_{\theta}k_2^\theta)+N_{\theta}+ B_{\theta}k_1^\theta}{R_{\theta}+D_{\theta}^2k_2^\theta},\frac{\alpha}{R_{\theta}+D_{\theta}^2k_2^\theta)} ).
\end{aligned}    
\end{equation}
where $\{X_\theta^*(t),t\geq0\}$ is the associated state process of the exploratory problem, starting from the state $x$, when $\pi^*$ is applied. It follows from straightforward calculation that
$$\int_{\mathbb{R}}\pi^*_{t,\theta}(u)\ln\pi^*_{t,\theta}(u) \ du=\frac12 \ln(\frac{2\pi e\alpha}{R_\theta+D_\theta^2k_2^\theta} ).$$
The desired result now follows immediately from the expression of $V(x)$ in (a) and $V^{cl}(x)$ in (b).  
\end{proof}
It is interesting to note that the exploratory cost is independent of the market coefficient $\theta$ (or equivalent, $Q$). This means that the exploratory cost will not increase when external environment becomes worse.
\section{Conclusion}
In this paper, we propose a continuous-time entropy-regularized reinforcement learning framework under model uncertainty. This is a fully data-driven approach that first identifies the optimal policy distribution for any given model parameter, and then determines the worst-case model parameter distribution under a robust cost criterion.

Our main contribution lies in demonstrating the feasibility of solving such problems using reinforcement learning by leveraging Sion’s minimax theorem to exchange the order of optimization over the policy and maximization over the model parameter distribution. Furthermore, we specify the sufficient conditions under which Sion’s minimax theorem can be applied, including requirements on the state dynamics, the cost function, and the admissible set of policy distributions.

Finally, we study two linear-quadratic (LQ) problems where the model parameters follow either a Bernoulli or a uniform distribution, and derive the corresponding optimal policy distributions in each case.

\appendix
\section{The system of algebra equation.}
The HJB equation has multiple solutions. Applying a generic quadratic function ansatz $v(x) = \lambda^{\gamma}v(x_1) + (1-\lambda)^{\gamma}v(x_2) = \lambda^{\gamma}(\frac12 k_{22}x_1^2+k_{11}x_1+c)+(1-\lambda)^{\gamma}(\frac12 k_{22}x_2^2+k_{12}x_2+c)$ in \eqref{LQ1hjbsub} yields the system of algebraic equations
\begin{scriptsize}
\begin{equation}
\begin{aligned}
    \lambda(1-\lambda)[S_1S_2+(S_1B_2+C_2D_2S_1)k_{22} + (B_1B_2+ C_2D_2B_1+C_1C_2D_1D_2)k_{21}k_{22}]=0 
\end{aligned}    
\end{equation}
\begin{equation}
\begin{aligned}
    \lambda(1-\lambda)[S_1S_2+(S_2B_1+C_1D_1S_2)k_{21} + (B_1B_2+ C_1D_1B_2+C_1C_2D_1D_2)k_{21}k_{22}]=0 
\end{aligned}    
\end{equation}
\begin{equation}
\begin{aligned}
    &\lambda [2D_1^2(\lambda A_1 + \frac{\lambda}{2}C_1-\frac{\rho}{2})-\lambda B_1^2 - 2\lambda C_1D_1B_1 - \lambda C_1^2D_1^2]k_{21}^2 + 2(1-\lambda)D_2^2(\lambda A_1  + \frac{\lambda}{2}C_1-\frac{\rho}{2})k_{21}k_{22} +\\ & (\lambda^2 D_1^2L_1+2\widetilde{R}(\lambda A_1 +  \frac{\lambda}{2}C_1-\frac{\rho}{2}) - 2\lambda^2 S_1B_1 - 2\lambda^2 C_1D_1S_1)k_{21}+  \lambda(1-\lambda)D_2^2L_1k_{22} + \lambda\widetilde{R}L_1 - \lambda^2 S_1^2 = 0\label{LQ1q1}
\end{aligned}    
\end{equation}
\begin{equation}
\begin{aligned}
    &(1-\lambda)[2D_2^2((1-\lambda) A_2 + \frac{1-\lambda}{2}C_2-\frac{\rho}{2})- (1-\lambda) B_2^2 - 2(1-\lambda)C_2D_2B_2 - (1-\lambda) C_2^2D_2^2]k_{22}^2  +\\ & 2\lambda D_1^2((1-\lambda) A_2 +  \frac{1-\lambda}{2}C_2-\frac{\rho}{2})k_{21}k_{22} + [(1-\lambda)^2 D_2^2L_2+  2\widetilde{R}((1-\lambda) A_2 + \frac{1-\lambda}{2}C_2-\frac{\rho}{2}) - \\ &2(1-\lambda)^2 S_2B_2 - 2(1-\lambda)^2 C_2D_2S_2)k_{22}+\lambda(1-\lambda)D_1^2L_2k_{21} + (1-\lambda)\widetilde{R}L_2 -  (1-\lambda)^2 S_2^2 = 0\label{LQ1q2}
\end{aligned}    
\end{equation}
\begin{equation}
\begin{aligned}
    &[2(\widetilde{R} +  \lambda D_1^2k_{21} + (1-\lambda)D_2^2 k_{22} )(\rho -\lambda A_1+2\lambda^2B_1S_1+ 2\lambda^2 B_1^2k_{21}+2\lambda^2C_1D_1B_1k_{21}]k_{11} \\ & +
    2\lambda(1-\lambda)[B_2S_1+B_2B_1k_{22}+C_2D_2B_1k_{22}]k_{12}
    =
    \lambda\{2(\widetilde{R} +  \lambda D_1^2k_{21} + (1-\lambda)D_2^2 k_{22} )M_1 \\ &- 2S_1[\lambda N_1+ (1-\lambda)N_2] - 2N_1[\lambda B_1k_{21} + (1-\lambda)B_2k_{22}+\lambda C_1D_1k_{21}+(1-\lambda)C_2D_2k_{22}]\}
\end{aligned}    
\end{equation}
\begin{equation}
\begin{aligned}
    &2\lambda(1-\lambda)[B_1S_2+B_2B_1k_{21}+C_1D_1B_2k_{21}]k_{11} +
    [2(\widetilde{R} + \lambda D_1^2k_{21} + (1-\lambda)D_2^2 k_{22} )(\rho -(1-\lambda) A_2)+\\ &2(1-\lambda)^2B_2S_2+2(1-\lambda)^2 B_2^2k_{22}+2(1-\lambda)^2C_2D_2B_2k_{22}]k_{12} 
    =
    (1-\lambda)\{2(\widetilde{R} +  \lambda D_1^2k_{21} + (1-\lambda)D_2^2 k_{22} )M_2 \\ &- 2S_2[\lambda N_1+ (1-\lambda)N_2 ] - 2N_2[\lambda B_1k_{21} + (1-\lambda)B_2k_{22} +\lambda C_1D_1k_{21}+(1-\lambda)C_2D_2k_{22}]\}
\end{aligned}    
\end{equation}
\begin{equation}
\begin{aligned}    &2(\widetilde{R}+D_1^2k_{21}+(1-\lambda)D_2^2k_{22})\rho c:= -[(\lambda N_1 + (1-\lambda)N_2)^2  + 2(\lambda B_1k_{11}+(1-\lambda)B_2k_{12})(\lambda N_1 + (1-\lambda)N_2)+\\ & (\lambda B_1k_{11} + (1-\lambda)B_2k_{12})^2- \frac{\alpha}{2}\ln(\frac{e^2\pi\alpha}{\widetilde{R}+\frac{\lambda}{2}D_1^2k_{21}+\frac{1-\lambda}{2}D_2^2k_{22}})]
\end{aligned}    
\end{equation}
\end{scriptsize}
Cause the quadratic equations \eqref{LQ1q1}, \eqref{LQ1q2} have two roots respectively, this system has four sets of solutions, leading to four quadratic solutions to the HJB equations. The one given though \eqref{LQ1Ks} is the one of the four solutions.

\section{Stability of Stochastic Differential Equations}
In this appendix, we state some well-known results about the stability of SDEs for readers' convenience. First, consider the following forward SDEs on $[0,\infty),$
\begin{equation}
\begin{aligned}
    x(t) = x_0 + \int_0^t b(s, x(s)) ds+\int_0^t \sigma(s, x(s))dW(s),\label{sde appendix b}
\end{aligned}    
\end{equation}
where $b:[0,\infty)\times\Omega\times\mathbb{R}\rightarrow \mathbb{R}$ and $\sigma:[0,\infty)\times\Omega\times\mathbb{R}\rightarrow \mathbb{R}$ satisfy the conditions such that \eqref{sde appendix b} has a unique strong solution.\par
\begin{lemma}
\label{appendixB}
    Assume that there is a functions $V(t,x)\in C^{2,1}(\mathbb{R}^d\times[t_0,\infty);\mathbb{R}_+)$, and positive constants $c_1-c_3$, such that
$$c_1|x|^p\leq V(t,x)\leq c_2|x|^p,\quad \mathcal{L}V(t,x)\leq -c_3 V(t,x)$$ 
for all $(x,t)\in \mathbb{R}^d\times [t_0,\infty)$. Then 
$$\mathbb{E}|x(t)|^p\leq \frac{c_2}{c_1}|x_0|^p e^{-c_3(t-t_0)}, \quad t\geq t_0$$
for all $x_0\in \mathbb{R}^d$. In other words, the trivial solution of equation \eqref{sde appendix b}
 is \textit{p}th moment exponentially stable. 
\end{lemma} 
 For the proof of \textbf{Lemma \ref{appendixB}}, we refer the reader to \cite{MAO2011107}.

\bibliographystyle{elsarticle-num} 
\bibliography{references}

\end{document}